\documentclass[12pt]{article}
\usepackage{amsmath, amssymb, amsthm, graphicx, hyperref}
\newtheorem{thm}{Theorem}
\newtheorem{lem}[thm]{Lemma}
\newtheorem{prop}[thm]{Proposition}
\newtheorem{corr}[thm]{Corollary}
\newtheorem{conj}[thm]{Conjecture}

\newcommand{\re}{\operatorname{Re}}

\newcommand{\twoFone}{{}_2F_1}
\newcommand{\centeredtheta}{\sigma_3}
\newcommand{\rescaledcenteredtheta}{\widehat{\sigma}_3}

\title{The Taylor coefficients of the Jacobi theta constant $\theta_3$}
\author{Dan Romik\footnote{
Department of Mathematics, 
University of California, Davis, One Shields Ave., Davis, CA 95616, USA. Email:
\texttt{romik@math.ucdavis.edu}
}}

\begin{document}

\maketitle

\begin{abstract}
We study the Taylor expansion around the point $x=1$ of a classical modular form, the Jacobi theta constant $\theta_3$. This leads naturally to a new sequence $(d(n))_{n=0}^\infty=1,1,-1,51,849,-26199,\ldots$ of integers, which arise as the Taylor coefficients in the expansion of a related ``centered'' version of $\theta_3$. We prove several results about the numbers $d(n)$ and conjecture that they satisfy the congruence $d(n)\equiv (-1)^{n-1}\ (\textrm{mod }5)$ and other similar congruence relations.
\end{abstract}

\renewcommand{\thefootnote}{\fnsymbol{footnote}} 
\footnotetext{
% 30-XX Functions of a complex variable
% 30B10 Power series (including lacunary series)
}
\renewcommand{\thefootnote}{\arabic{footnote}}

\section{Introduction}

\subsection{The derivatives of the theta constant $\theta_3(x)$ at $x=1$}

The Jacobi theta constant (or ``thetanull'') function $\theta_3$ is defined by
\begin{equation} \label{eq:def-theta}
\theta_3(x) = \sum_{n=-\infty}^\infty e^{-\pi n^2 x} = 1 + 2 \sum_{n=1}^\infty e^{-\pi n^2 x} \qquad (\re{x}>0),
\end{equation}
and is one of the most classical and well-studied objects in number theory, being intimately tied to the study of integer partitions, representations of integers as sums of squares, the Riemann zeta function, modular forms, and much else. In this paper we will study the Taylor coefficients of $\theta_3(x)$ around $x=1$, that is, the sequence of derivatives $\theta_3^{(n)}(1)$. This is motivated by the fact that $\theta_3(x)$ satisfies the well-known modular transformation property
\begin{equation}
\label{eq:theta-modular}
\theta_3\left(\frac1x\right) = \sqrt{x} \,\theta_3(x), 
\end{equation}
which makes the point $x=1$ left fixed under $x\mapsto 1/x$ a natural point near which to try to understand the local behavior of $\theta_3(x)$ (in particular as a way of gaining new insight into the content of \eqref{eq:theta-modular}). Two additional relevant motivating facts from the literature are the explicit evaluation
\begin{equation}
\label{eq:theta-eval}
\theta_3(1) 
%= \frac{\pi^{1/4}}{\Gamma\left(\tfrac34\right)}, 
= \frac{\Gamma\left(\tfrac14\right)}{\sqrt{2}\pi^{3/4}},
\end{equation}
(where $\Gamma(\cdot)$ is the Euler gamma function) ---
see \cite[p.~325]{ram5} --- and the relation
$$
\theta_3'(1) = -\frac14 \theta_3(1),
$$
which follows on differentiating \eqref{eq:theta-modular} and setting $x=1$ (as pointed out for example in \cite[p.~17]{edwards}). It seems natural to ask about the values of the higher derivatives $\theta_3''(1), \theta_3'''(1), \ldots$, but this problem does not appear to have been previously addressed in the literature (although related questions have been considered; see Section~\ref{subsec:previouswork}). It turns out that these values can be expressed in terms of $\theta_3(1)$ and an additional constant
$$
\Omega = \frac{\Gamma\left(\tfrac14\right)^8}{32 \pi^4}.
$$
For example, we will show that the first few are given by
\begin{align*}
\theta_3''(1) &= \phantom{-}\frac{1}{16}\theta_3(1)\left(3 + \Omega\right), 
\\[3pt]
\theta_3'''(1) &= -\frac{1}{64}\theta_3(1)\left(15 +15 \Omega\right), 
\\[3pt]
\theta_3^{(4)}(1) &= \phantom{-}\frac{1}{256}\theta_3(1)\left(105+210 \Omega - \Omega^2\right).
\end{align*}

These formulas are special cases of a more general result, which introduces a curious new object: an integer sequence $(d(n))_{n=0}^\infty$ that will be our main object of study.

\begin{thm}[Taylor coefficients of $\theta_3(x)$]
\label{thm:theta-taylor}
There exists a sequence of integers $$(d(n))_{n=0}^\infty = 1,1,-1,51,849, -26199, 1341999, \ldots$$ with the property that
\begin{equation} 
\label{eq:theta-taylor}
\theta_3^{(n)}(1) = \theta_3(1) \cdot \frac{(-1)^n }{4^n} \sum_{k=0}^{\lfloor n/2\rfloor} \frac{(2n)!}{2^{n-2k}(4k)!(n-2k)! } d(k) \Omega^k  \qquad (n\ge 0).
\end{equation}
\end{thm}

The initial values $(d(n))_{n=0}^{20}$ are tabulated in the Appendix. Later we will present two different methods for computing the $d(n)$'s.

\subsection{A centered version of $\theta_3$ and the integer sequence $(d(n))_{n=0}^\infty$}

While \eqref{eq:theta-taylor} gives an answer of sorts to the original question we posed, the relative messiness of the expansion on the right-hand side of \eqref{eq:theta-taylor} serves to obscure somewhat the underlying phenomena at play and the true significance of the integer sequence $(d(n))_{n=0}^\infty$. This significance is made more apparent by a change of coordinates: define a function $\centeredtheta(z)$ of a complex variable by
\begin{equation} \label{eq:def-sigma}
\centeredtheta(z) = \frac{1}{\sqrt{1+z}} \theta_3\left(\frac{1-z}{1+z}\right) \qquad (|z|<1).
\end{equation}
A trivial calculation shows that the modular transformation relation \eqref{eq:theta-modular} is equivalent to the simpler relation
\begin{equation} \label{eq:sigma-even} 
\centeredtheta(-z) = \centeredtheta(z),
\end{equation}
that is, the statement that $\centeredtheta(z)$ is an even function. (Geometrically, note that the change of coordinates $x\mapsto z=\frac{1-x}{1+x}$ maps the right half-plane $\{ \re(x)>0 \}$ conformally to the unit disk, and that under this change of coordinates the M\"obius inversion $x\mapsto 1/x$ translates to the simple reflection $z\mapsto -z$.) Thus, we can think of $\centeredtheta(z)$ as a version of $\theta_3$ that is ``centered'' around $x=1$. This point of view naturally suggests considering the Taylor expansion of $\centeredtheta(z)$ around $z=0$ (the point in the $z$-plane whose preimage under the coordinate change is $x=1$), which encodes the same information as \eqref{eq:theta-taylor} but packaged in a different way. Our next result shows that the coefficients in this expansion are, modulo trivial factors, the integers $d(n)$.

\begin{thm}[Taylor expansion of $\centeredtheta(z)$]
\label{thm:sigma-taylor}
The function $\centeredtheta(z)$ has the Taylor expansion 
\begin{equation}\label{eq:sigma-taylor}
\centeredtheta(z) = \theta_3(1)\sum_{n=0}^\infty \frac{d(n)}{(2n)!} \Phi^n z^{2n} \qquad (|z|<1),
\end{equation}
where we denote
$ \displaystyle
\Phi = \frac\Omega4  = \frac{\Gamma\left(\tfrac14\right)^8}{128 \pi^4}
$.
\end{thm}

In addition to the fact that it clarifies the role of the numbers $d(n)$, the series expansion \eqref{eq:sigma-taylor} also has the advantage that it converges for all $|z|<1$, whereas the Taylor series $\theta_3(x)=\sum_{n=0}^\infty \frac{\theta_3^{(n)}(1)}{n!}(x-1)^n$ around $x=1$ has radius of convergence $1$ and consequently converges only on a small part of the region $\{ \re(x)>0\} $ where $\theta_3(x)$ is defined. Equivalently, converting \eqref{eq:sigma-taylor} back to a formula for $\theta_3(x)$ via the relation 
\begin{equation}
\label{eq:sigma-theta}
\theta_3(x) = \sqrt{\frac{2}{1+x}}\,\centeredtheta\left(\frac{1-x}{1+x}\right)
\end{equation} inverse to \eqref{eq:def-sigma}, we get that $\theta_3(x)$ has the expansion
$$
\theta_3(x) = \theta_3(1) \frac{\sqrt{2}}{\sqrt{1+x}}\sum_{n=0}^\infty \frac{d(n)}{(2n)!} \Phi^n \left(\frac{1-x}{1+x}\right)^{2n} \qquad (\re{x}>0),
$$
which converges throughout the domain of $\theta_3(x)$.

Note as well that the expansion \eqref{eq:sigma-taylor} has the even symmetry \eqref{eq:sigma-even} of $\centeredtheta(z)$ built into it, thus neatly encapsulating the modular transformation property~\eqref{eq:theta-modular}.

\subsection{A generating function identity for the $d(n)$'s}

Having motivated the definition of the sequence of coefficients $(d(n))_{n=0}^\infty$, our main goal is to understand their behavior and derive interesting formulas involving them; as corollaries we will obtain proofs of Theorems~\ref{thm:theta-taylor}~and~\ref{thm:sigma-taylor}, and practical algorithms for computing the $d(n)$'s. The most fundamental identity, which we formulate now, relates the generating function of the sequence $(d(n))_{n=0}^\infty$ to two specializations of the Gauss hypergeometric function $\twoFone$. Recall that $\twoFone$ is defined by
$$
\twoFone(a,b;c;z) = \sum_{n=0}^\infty \frac{1}{n!} \frac{(a)_n (b)_n}{(c)_n} z^n,
$$
where
$$ (x)_n = x(x+1)\ldots (x+n-1) $$
is the Pochhammer symbol.

\begin{thm}[Generating function identity for the numbers $d(n)$]
\label{thm:dofn-genfun}
Define functions $U(t), V(t)$ by
\begin{align}
U(t) &= \frac{\twoFone(\tfrac34,\tfrac34;\tfrac32;4t)}{\twoFone(\tfrac14,\tfrac14;\tfrac12;4t)}, 
\label{eq:def-u}
\\[4pt]
V(t) &= \sqrt{\twoFone(\tfrac14,\tfrac14;\tfrac12;4t)}.
\label{eq:def-v}
\end{align}
Then $(d(n))_{n=0}^\infty$ is the unique sequence of real numbers satisfying the generating function identity
\begin{equation} \label{eq:dofn-genfun}
\sum_{n=0}^\infty \frac{d(n)}{2^n (2n)!} t^n \,U(t)^{2n} = V(t) \qquad (|t|<1).
\end{equation}
\end{thm}

\subsection{Taylor coefficients of modular forms and connections to previous work}

\label{subsec:previouswork}

The results of this paper can be viewed as part of the much broader theme of considering the Taylor coefficients of general modular forms around complex multiplication points. This point of view seems to have its roots in the work of Shimura \cite{shimura}, and was later considered by Villegas and Zagier \cite{villegas-zagier1, villegas-zagier2} and others \cite{datsk, larson-smith, voight-willis}; see also Sections 5.1~and~6.3~of~\cite{zagier123}. Our idea of considering the ``centered'' version of a modular form, and the connection between the Taylor expansions of the centered and non-centered version, are discussed in \cite{zagier123} in that broader setting.

While our point of view has some overlap with the existing literature and therefore in some sense our results can be viewed as not much more than a special case of a more general theory, the emphasis in the current work is on considering the sequence $(d(n))_{n=0}^\infty$ as an interesting sequence of integers and studying its properties, deriving explicit algorithms for its computation, etc. Moreover, our impression is that in the existing literature scant attention has been paid to the fact that sequences of Taylor coefficients of modular forms can give rise to sequences of integers (as opposed to algebraic numbers) that are worth studying for their own sake. We believe that this point of view deserves to be emphasized and further explored systematically; see also the discussion in the section on open problems (Section~\ref{sec:openproblems} below).

\subsection{Structure of the paper}

The rest of the paper is organized as follows. In Section~\ref{sec:proof-genfun} we prove Theorem~\ref{thm:dofn-genfun}, starting from \eqref{eq:sigma-taylor}, which we take as our definition of $d(n)$. In Section~\ref{sec:dofn-recurrence} we will show how Theorem~\ref{thm:dofn-genfun} can be reformulated into a recurrence relation for $d(n)$. In Section~\ref{sec:proof-theta-taylor} we prove Theorem~\ref{thm:theta-taylor}. In Section~\ref{sec:infsums} we recast some of the results as explicit formulas for some infinite series. In Section~\ref{sec:sigma-ode} we show that the centered function $\centeredtheta(z)$ satisfies a nonlinear third-order differential equation that is very similar to the one satisfied by $\theta_3(x)$. In Section~\ref{sec:congruences} we discuss some conjectural congruences satisfied by the sequence of coefficients $((d(n))_{n=0}^\infty$. Section~\ref{sec:openproblems} lists some open problems.

\subsection*{Acknowledgements}

The author thanks Robert Scherer, Christian Krattenthaler, Tanguy Rivoal, David Broadhurst, Peter Paule, Yiangjie Ye, Doron Zeilberger, Craig Tracy, David Bailey and Bill Gosper for helpful discussions during the preparation of this manuscript. Some of these discussions took place during the author's visit to the Erwin Schr\"odinger Institute (ESI) in November 2017; the author is grateful to ESI for its support and hospitality. The author also thanks the anonymous referee for suggesting useful corrections.

This material is based upon work supported by the National Science Foundation under Grant No.~DMS-1800725.

\section{Proof of Theorem~\ref{thm:dofn-genfun} assuming Theorem~\ref{thm:sigma-taylor}}

\label{sec:proof-genfun}

For the remainder of the paper, we take the Taylor expansion \eqref{eq:sigma-taylor} as the definition of the sequence $(d(n))_{n=0}^\infty$. Our goal in this section is to prove the generating function identity \eqref{eq:dofn-genfun}. The main tool we will use is a well-known identity of Jacobi that relates $\theta_3(x)$ to the complete elliptic integral of the first kind. Specifically, define
\begin{equation} \label{eq:goft-def}
G(t) = \twoFone(\tfrac12,\tfrac12;1;t) = \sum_{n=0}^\infty \binom{2n}{n}^2 \left(\frac{t}{16}\right)^n,
\end{equation}
which can also be written as $G(t) = \frac{2}{\pi} K(\sqrt{t})$,
where
$$ K(k) = \int_0^1 \frac{dx}{\sqrt{(1-x^2)(1-k^2 x^2)}} $$
is the complete elliptic integral of the first kind. Jacobi's identity states \cite[p.~101]{ram3} that
\begin{equation} \label{eq:theta-jacobi-hypgeom}
\theta_3\left( \frac{G(1-t)}{G(t)} \right) = \sqrt{G(t)} \qquad (0<t<1)
\end{equation}
(see also \cite[pp.~479, 499]{whittaker-watson}).
Without doing any explicit computations, one can see immediately from \eqref{eq:theta-jacobi-hypgeom} that the local behavior of $\centeredtheta(z)$ near $z=0$, which (as we noted in the introduction) is equivalent through the change of variables $z=\frac{1-x}{1+x}$ to the local behavior of $\theta_3(x)$ near $x=1$, also becomes equivalent through the further change of variables $x=G(1-t)/G(t)$ to the local behavior of $\theta_3(x)=\sqrt{G(t)}$ near $t=1/2$. Thus, it seems plausible that~\eqref{eq:theta-jacobi-hypgeom} could be used to obtain information about the Taylor expansion of $\centeredtheta(z)$ near $z=0$. However, this requires understanding the Taylor expansion of $G(t)$ around $t=1/2$ (where note that it is defined as a Taylor series around $t=0$). The next result bridges that gap.

\begin{prop} The Taylor expansion of $G(t)$ around $t=1/2$ is given by
\begin{align} 
\nonumber G(t) &= \frac{1}{2\pi^{3/2}} \sum_{n=0}^\infty \frac{4^n \,\Gamma\left(\frac{n}{2}+\tfrac14\right)^2}{n!} \left(t-1/2\right)^n
\\ &= 
\frac{1}{2\pi^{3/2}}
\left( \Gamma\left(\tfrac14\right)^2 Q(t-1/2)+ 
4 \Gamma\left(\tfrac34\right)^2 P(t-1/2) \right),
\label{eq:goft-taylor}
\end{align}
where we define
\begin{align*} 
P(u) &= \sum_{m=0}^\infty \frac{(3\cdot 7\cdot 11\cdot \ldots \cdot (4m-1))^2}{(2m+1)!} u^{2m+1} = u\cdot \twoFone(\tfrac34,\tfrac34;\tfrac32;4u^2), \\
Q(u) &= \sum_{m=0}^\infty \frac{(1\cdot 5\cdot 9\cdot \ldots \cdot (4m-3))^2}{(2m)!} u^{2m} = \twoFone(\tfrac14,\tfrac14;\tfrac12;4u^2).
\end{align*}

\end{prop}

\begin{proof}
The second equality in \eqref{eq:goft-taylor} is immediate upon separating the infinite sum on the first row into its odd and even (as functions of the variable $s=t-1/2$) components. It is therefore enough to show that 
$G(t)$ coincides with the function 
$$\gamma(t):= \frac{1}{2\pi^{3/2}} \sum_{n=0}^\infty \frac{4^n}{n!} \Gamma\left(\frac{n}{2}+\tfrac14\right)^2\left(t-1/2\right)^n.
$$ 
But recall that $G(t)=\twoFone(\tfrac12,\tfrac12;1;t)$, being a specialization of a Gauss $\twoFone(a,b;c;t)$ function, satisfies the hypergeometric differential equation \cite{bailey}
\begin{equation} \label{eq:twofone-ode}
t(1-t) g''(t) + (1-2t)g'(t) - \tfrac14 g(t) = 0.
\end{equation}
It is likewise easy to verify directly from its definition that $\gamma(t)$ is also a solution to \eqref{eq:twofone-ode}. Since $G(t)$ and $\gamma(t)$ both satisfy the same second-order linear ordinary differential equation, it suffices to prove that $G(1/2) = \gamma(1/2)$ and $G'(1/2)=\gamma'(1/2)$. Now, the definition of $\gamma(t)$ gives that 
\begin{align*}
\gamma(1/2)&=\frac{1}{2\pi^{3/2}}\Gamma\left(\tfrac14\right)^2 Q(0)=\frac{1}{2\pi^{3/2}}\Gamma\left(\tfrac14\right)^2, \\
%$$ 
%and
%$$
\gamma'(1/2)&=\frac{1}{2\pi^{3/2}}\cdot 4\Gamma\left(\tfrac34\right)^2 P'(0)=\frac{2}{\pi^{3/2}}\cdot \Gamma\left(\tfrac34\right)^2.
\end{align*} 
These numbers indeed coincide with $G(1/2)$ and $G'(1/2)$, respectively, since the value $G(1/2)=\frac{1}{2\pi^{3/2}}\Gamma\left(\tfrac14\right)^2$ is classically known, \cite[p.~11]{bailey} and the evaluation $G'(1/2)=\frac{2}{\pi^{3/2}}\Gamma\left(\tfrac34\right)^2$ is similarly easy to obtain using standard identities (specifically, the differential equation $\frac{dK(k)}{dk} = \frac{E(k)}{k(1-k^2)}-\frac{K(k)}{k}$ and Legendre's relation 
$ K(k) E(k')+E(k)K(k')-K(k)K(k')=\pi/2$, where $k'=\sqrt{1-k^2}$, and $E(k)$ denotes the complete elliptic integral of the second kind).
\end{proof}

We are now ready to prove Theorem~\ref{thm:dofn-genfun}. Let $t$ and $x$ be related by $x=G(1-t)/G(t)$. Then making use of \eqref{eq:sigma-taylor}, \eqref{eq:sigma-theta} and \eqref{eq:theta-jacobi-hypgeom} we have that
\begin{align*}
\sqrt{G(t)} &= \theta_3\left(\frac{G(1-t)}{G(t)}\right) = \theta_3(x) = \frac{\sqrt{2}}{\sqrt{1+x}} \centeredtheta\left(\frac{1-x}{1+x}\right) \\ &= \sqrt{2} \frac{\sqrt{G(t)}}{\sqrt{G(t)+G(1-t)}} \centeredtheta\left( \frac{G(t)-G(1-t)}{G(t)+G(1-t)} \right)
\\ &= \sqrt{2} 
\frac{\sqrt{G(t)}}{\sqrt{G(t)+G(1-t)}} 
\cdot \theta_3(1) 
\sum_{n=0}^\infty \frac{d(n)}{(2n)!} \Phi^n \left( \frac{G(t)-G(1-t)}{G(t)+G(1-t)} \right)^{2n},
\end{align*}
or, equivalently,
\begin{equation} \label{eq:sqrt-gt-goneminus}
\sqrt{G(t)+G(1-t)}
=
\sqrt{2} 
\cdot \theta_3(1) 
\sum_{n=0}^\infty \frac{d(n)}{(2n)!} \Phi^n \left( \frac{G(t)-G(1-t)}{G(t)+G(1-t)} \right)^{2n}.
\end{equation}
But now note that from \eqref{eq:goft-taylor}, we have that
\begin{align*}
G(t)-G(1-t) &=
\frac{4}{\pi^{3/2}}\Gamma\left(\tfrac34\right)^2 P(t-1/2)
\\ &= 
\frac{4}{\pi^{3/2}} \Gamma\left(\tfrac34\right)^2 (t-1/2)\cdot \twoFone(\tfrac34,\tfrac34;\tfrac32; 4(t-1/2)^2 ),
\\
G(t) + G(1-t) &=
\frac{1}{\pi^{3/2}}\Gamma\left(\tfrac14\right)^2 Q(t-1/2)
\\ &= 
\frac{1}{\pi^{3/2}}  \Gamma\left(\tfrac14\right)^2\, \twoFone(\tfrac14,\tfrac14;\tfrac12; 4(t-1/2)^2 ),
\end{align*}
so that, denoting $s=(t-1/2)^2$, we can rewrite \eqref{eq:sqrt-gt-goneminus} as
\begin{align*}
\frac{1}{\pi^{3/4}} &\Gamma\left(\tfrac14\right) \sqrt{
\twoFone(\tfrac14,\tfrac14;\tfrac12; 4s)}
\\&= \sqrt{2}\,\theta_3(1)\sum_{n=0}^\infty \frac{d(n)}{(2n)!} \Phi^n 
\left(\frac{4\Gamma\left(\tfrac34\right)^2}{\Gamma\left(\tfrac14\right)^2} \right)^{2n} s^n \left(
\frac{\twoFone(\tfrac34,\tfrac34;\tfrac32; 4s)}{\twoFone(\tfrac14,\tfrac14;\tfrac12; 4s)}
\right)^{2n}.
\end{align*}
Equivalently, referring to \eqref{eq:def-u}--\eqref{eq:def-v}, we have that
$$ V(s) = C_1 \sum_{n=0}^\infty \frac{d(n)}{2^n (2n)!} C_2^n s^n U(s)^{2n}, $$
where we define constants
$$
C_1 = \frac{\sqrt{2} \,\pi^{3/4} \,\theta_3(1)}{\Gamma\left(\tfrac14\right)}, 
\qquad 
C_2 = 2 \Phi  \left(\frac{4\Gamma\left(\tfrac34\right)^2}{\Gamma\left(\tfrac14\right)^2} \right)^2.
$$
Now, referring to \eqref{eq:theta-eval} we see that $C_1=1$. Similarly, recalling the standard relation $\Gamma\left(\tfrac14\right)\Gamma\left(\tfrac34\right)=\sqrt{2}\pi$,
one can easily check that $C_2=1$.
This finishes the proof of \eqref{eq:dofn-genfun}.
\qed

\section{A recurrence relation for $d(n)$}

\label{sec:dofn-recurrence}

We now reformulate the generating function identity \eqref{eq:dofn-genfun} as an explicit recurrence relation for $d(n)$.

\begin{lem}
The Taylor expansion of $U(t)$ (defined in \eqref{eq:def-u}) around $t=0$ is given by
$$
U(t) = \sum_{n=0}^\infty \frac{u(n)}{(2n+1)!} t^n,
$$
where $(u(n))_{n=0}^\infty = 1,6,256,28560,\ldots$ is a sequence of integers which satisfy for $n\ge1$ the recurrence relation
\begin{equation} \label{eq:uofn-recurrence}
u(n) = \Big(3\cdot 7\cdot 11\cdot \ldots \cdot (4n-1)\Big)^2 - \sum_{m=0}^{n-1} \binom{2n+1}{2m+1}\Big(1\cdot 5\cdot 9\cdot \ldots \cdot (4(n-m)-3)\Big)^2 u(m).
\end{equation}
\end{lem}

\begin{proof}
The relation \eqref{eq:uofn-recurrence} is immediate from the definition of $u(n)$ upon rewriting \eqref{eq:def-u} in the form 
\begin{equation}
\label{eq:uofn-taylor}
U(t) \twoFone(\tfrac14,\tfrac14;\tfrac12;4t)
= \twoFone(\tfrac34,\tfrac34;\tfrac32;4t)
\end{equation}
and equating the Taylor coefficients on both sides.
The recurrence then implies that $u(n)$ is an integer.
\end{proof}

\begin{lem}
The Taylor expansion of $V(t)$ around $t=0$ is given by
\begin{equation} \label{eq:vofn-taylor}
V(t) = \sum_{n=0}^\infty \frac{v(n)}{2^n (2n)!} t^n,
\end{equation}
where $(v(n))_{n=0}^\infty = 1,1,47,7395,2453425,\ldots$ are integers which satisfy for $n\ge1$ the recurrence relation
\begin{equation} \label{eq:vofn-recurrence}
v(n) = 2^{n-1}\Big(1\cdot 5\cdot 9\cdot \ldots \cdot (4n-3)\Big)^2 - \frac12\sum_{m=1}^{n-1} \binom{2n}{2m}
v(m)v(n-m).
\end{equation}
\end{lem}

\begin{proof}
To prove \eqref{eq:vofn-recurrence}, rewrite \eqref{eq:def-u} in the form 
$$ V(t)^2 = \twoFone(\tfrac14,\tfrac14;\tfrac12;4t) $$
and equate coefficients. It is also not difficult to see that $v(n)$ is an integer, by exploiting symmetry to rewrite \eqref{eq:vofn-recurrence} in the form
\begin{align} 
\nonumber
v(n) &= 2^{n-1}\Big(1\cdot 5\cdot 9\cdot \ldots \cdot (4n-3)\Big)^2 - \sum_{1\le m <\lfloor \frac{n}{2}\rfloor }\binom{2n}{2m}v(m)v(n-m)
\\ & \qquad - \frac12 \chi_{\{n\textrm{ even}\}} \binom{2n}{n} v(n/2)^2 
\label{eq:vofn-recurrence2}
\end{align}
(where $\chi_{\{n\textrm{ even}\}}=1$ if $n$ is even, $0$ otherwise)
and noting that $\binom{2n}{n}$ is always an even integer.
\end{proof}

\begin{thm}[Recurrence relation for $d(n)$]

\label{thm:recurrence-dofn}

The numbers $(d(n))_{n=0}^\infty$ satisfy the recurrence relation
\begin{equation}
\label{eq:dofn-recurrence}
d(n) = v(n) - \sum_{k=1}^{n-1} r(n,k) d(k) \qquad (n\ge1),
\end{equation}
where $(r(n,k))_{1\le k\le n}$ is a triangular array of integers defined by
\begin{equation}
\label{eq:rofnk-def}
r(n,k) = 2^{n-k} \frac{(2n)!}{(2k)!} [t^{n-k}]U(t)^{2k} \qquad (1\le k\le n).
\end{equation}
(Here, $[t^j]f(t)$ refers as usual to the $j$th coefficient $c_j$ in a power series $f(t)=\sum_{m=0}^\infty c_m t^m$.)
\end{thm}

\begin{proof} Thinking of \eqref{eq:rofnk-def} as also defining $r(n,k)$ when $k=0$, 
by comparing coefficients of like powers of $t$ in \eqref{eq:dofn-genfun} and \eqref{eq:vofn-taylor} we get that the equation
$$ v(n) = \sum_{k=0}^{n} r(n,k) d(k) $$
holds. Observing that $r(n,0)=0$ (if $n\ge1$) and $r(n,n)=1$ gives \eqref{eq:dofn-recurrence}.

It remains to prove that $r(n,k)$ is an integer for all $1\le k\le n$. This is a consequence of the exponential formula from combinatorics \cite[Ch.~5]{stanley-ec2}, which gives a combinatorial interpretation for the coefficients $b_{n,k}$ in the power series expansion
$$ 
\exp\left(t\sum_{j=1}^\infty \frac{a_j}{j!} x^j \right) =
\sum_{n,k=0}^\infty \frac{b_{n,k}}{n!} x^n t^k,
$$
and in particular implies that if all the $a_j$'s are integers then so are the $b_{n,k}$'s. Furthermore, $b_{n,k}$ is given explicitly by
\begin{equation} \label{eq:expform-coeff}
b_{n,k} = \frac{n!}{k!} [x^n] \left( \sum_{j=0}^\infty \frac{a_j}{j!}x^j\right)^k.
\end{equation}
To see why this is precisely what we need, observe that 
\eqref{eq:rofnk-def} can be rewritten as
\begin{align*}
r(n,k) &= 2^{n-k} \frac{(2n)!}{(2k)!} [t^n]
\left( \sum_{j=0}^\infty \frac{u(j)}{(2j+1)!} t^{j+1/2} \right)^{2k}
\\
&=
2^{n-k} \frac{(2n)!}{(2k)!} [s^{2n}]
\left( \sum_{j=0}^\infty \frac{u(j)}{(2j+1)!} s^{2j+1} \right)^{2k},
\end{align*}
which coincides with $2^{n-k} b_{2n,2k}$ in \eqref{eq:expform-coeff} if the coefficient sequence $(a_j)_{j=1}^\infty$ is defined by $a_{2j}=0$, $a_{2j+1}=u(j)$.
\end{proof}

\begin{corr} $d(n)$ is an integer for all $n\ge0$.
\end{corr}

Here are the first few entries in the array of numbers $(r(n,k))_{1\le k\le n}$ (formatted in matrix notation as the lower triangular part of an infinite matrix):
$$
(r(n,k))_{1\le k\le n}=
\begin{pmatrix}
 1
\\
48 & 1
\\
7584 & 240 & 1
\\
2515468 & 97664 & 672 & 1
\\
1432498176 & 63221760 & 560448 & 1440 & 1 \\
\vdots & \vdots & \vdots & \vdots & \vdots & \ddots
\end{pmatrix}
$$

The formulas in this section are implemented in a Maple software package written by Doron Zeilberger. \cite{zeilberger}

\section{Proof of Theorem~\ref{thm:theta-taylor}}

\label{sec:proof-theta-taylor}

First, note that for any integer $n\ge0$, we have the generating function identity
$$
(1-4u)^{-(4n+1)/2}=
\frac{(2n)!}{(4n)!} \sum_{m=2n}^\infty
\frac{(2m)!}{m!(m-2n)!}
u^{m-2n} \qquad (|u|<1/4)
$$
(apply the binomial theorem, or start with the the case $n=0$, which is the standard generating function identity 
$
\sum_{m=0}^\infty \binom{2m}{m} u^m = \frac{1}{\sqrt{1-4u}},
$
and differentiate $2n$ times).

Now, still working with the expansion \eqref{eq:sigma-taylor} as our definition of the sequence $(d(n))_{n=0}^\infty$, we apply \eqref{eq:sigma-theta} to get that
\begin{align*}
\frac{\theta_3(x)}{\theta_3(1)} 
&= 
\frac{\sqrt{2}}{\sqrt{1+x}}
\sum_{n=0}^\infty \frac{d(n)}{(2n)!} \Phi^n \left(
\frac{1-x}{1+x}\right)^{2n}
\\
&=
\frac{\sqrt{2}}{\sqrt{1+x}}
\sum_{n=0}^\infty \frac{d(n)}{(2n)!} \left(\frac{\Omega}{4}\right)^n \left(
\frac{1-x}{2}\right)^{2n}
\left(
\frac{1+x}{2}\right)^{-2n}
\\
&=
\sum_{n=0}^\infty \frac{4^n d(n)}{(2n)!} \Omega^n \left(
\frac{1-x}{8}\right)^{2n}
\left(
\frac{1+x}{2}\right)^{-(4n+1)/2}
\\
&=
\sum_{n=0}^\infty \frac{4^n d(n)}{(4n)!} \Omega^n \left(
\frac{1-x}{8}\right)^{2n}
\sum_{m=2n}^\infty
\frac{(2m)!}{m!(m-2n)!}
\left(\frac{1-x}{8}\right)^{m-2n}
\\
&=
\sum_{n=0}^\infty \frac{4^n d(n)}{(4n)!} \Omega^n 
\sum_{m=2n}^\infty
\frac{2m!}{m!(m-2n)!}
\left(\frac{1-x}{8}\right)^{m}
\\
&=
\sum_{m=0}^\infty \frac{(-1)^m}{8^m m!}  \left(
\sum_{n=0}^{\lfloor m/2\rfloor} 
\frac{4^n(2m)!}{(4n)!(m-2n)! } d(n) \Omega^n
\right)(x-1)^m.
\end{align*}
Equating the coefficient of $(x-1)^m$ in the last expression to $\frac{\theta_3^{(m)}(1)}{\theta_3(1)\,m!}$ gives~\eqref{eq:theta-taylor}. Since we already proved in the previous section that the $d(n)$'s are integers, this completes the proof of Theorem~\ref{thm:theta-taylor}.
\qed

\section{Some infinite sums}

\label{sec:infsums}

In this section we apply our results to prove explicit formulas for several interesting infinite series.

\begin{prop}
For any integer $k\ge 0$ we have
\begin{equation} \label{eq:infseries-thetaderiv}
\sum_{n=-\infty}^\infty n^{2k}e^{-\pi n^2} = 
\frac{\pi^{1/4}}{\Gamma\left(\tfrac34\right)}
\cdot \frac{1}{(4\pi)^k} \sum_{j=0}^{\lfloor k/2\rfloor} \frac{(2k)!}{2^{k-2j} (4j)!(k-2j)!} d(j) \Omega^j,
\end{equation}
(with the convention that $0^0 = 1$ being used to interpret the term $n=0$ on the left-hand side in the case $k=0$).
\end{prop}

\begin{proof}
This is simply a reformulation of \eqref{eq:theta-taylor}, expressing the derivative $\theta_3^{(n)}(1)$ explicitly as an infinite sum and moving the factor $(-\pi)^{k}$ from the differentiation of $e^{-\pi n^2 x}$ to the right-hand side.
\end{proof}

By replacing the monomial $n^{2k}$ on the left-hand side of \eqref{eq:infseries-thetaderiv} with a Hermite polynomial, the right-hand side is simplified considerably.

\begin{prop}
Let
$$ H_m(x) = (-1)^n e^{x^2}\frac{d^m}{dx^m} e^{-x^2} $$
denote the $m$th Hermite polynomial.
Then we have the identity
\begin{equation} \label{eq:infsum-hermite}
\sum_{n=-\infty}^\infty e^{-\pi n^2} H_{2k}(\sqrt{2\pi} n)
= \begin{cases}
\theta_3(1) 4^k \Phi^{k/2} d(k/2) & \textrm{$k$ even},
\\
0 & \textrm{$k$ odd},
\end{cases}
\qquad (k\ge 0).
\end{equation}
\end{prop}

\begin{proof}
Combining \eqref{eq:def-theta}, \eqref{eq:def-sigma} and \eqref{eq:sigma-taylor} gives that
the expression on the right-hand side of \eqref{eq:infsum-hermite} is equal to
$$ \sum_{n=-\infty}^\infty
\frac{d^{k}}{dz^{k}}_{\raisebox{2pt}{\big|}z=0} \left[
\frac{1}{\sqrt{1+z}} \exp\left(-\pi n^2 \frac{1-z}{1+z}\right)
\right].
$$
So we see that proving \eqref{eq:infsum-hermite} reduces to showing that
$$
\frac{d^{k}}{dz^{k}}_{\raisebox{2pt}{\big|}z=0} \left[
\frac{1}{\sqrt{1+z}} \exp\left(-X \frac{1-z}{1+z}\right)
\right]
=
4^{-k} e^{-X} H_{2k}(\sqrt{2X}),
$$
where $X>0$ is a parameter.
It is easy to see that this is equivalent to the bivariate generating function identity
\begin{equation} \label{eq:hermite-genfun1}
\frac{1}{\sqrt{1+z}} \exp\left(\frac{2z X}{1+z}\right)
=
\sum_{m=0}^\infty \frac{1}{m!} H_{2m}(\sqrt{2X}) \left(\frac{z}{4}\right)^m.
\end{equation}
This in turn is a (trivial rescaling of a) known identity; see, e.g., \cite[Eq.~(3.5)]{wunsche}.
\end{proof}

Note that \eqref{eq:infsum-hermite} gives another way to evaluate the $d(n)$'s numerically.

\section{A differential equation for $\centeredtheta(z)$}

\label{sec:sigma-ode}

Jacobi found that $\theta_3(x)$ satisfies the ordinary differential equation
\begin{align}
\big(
y^2 y''' - 15 y \,
y' \, y'' +
30 (y')^3
\big)^2
+ 32 \left( y\, y'' - 3 (y')^2\right)^3
=
\pi^2 y^{10} \left( y\, y'' - 3 (y')^2\right)^2.
\label{eq:theta-ode}
\end{align}
a nonlinear, third-order equation \cite{ohyama}. As a possible alternative approach to studying the Taylor coefficients $d(n)$, we use this to derive an ordinary differential equation satisfied by the associated function $\centeredtheta(z)$. Interestingly, the equation $\centeredtheta(z)$ satisfies is almost identical to \eqref{eq:theta-ode}.

\begin{thm}
The function $\centeredtheta(z)$ satisfies the ordinary differential equation
\begin{equation}
\big(
y^2 y''' - 15 y \,
y' \, y'' +
30 (y')^3
\big)^2
+ 32 \left( y\, y'' - 3 (y')^2\right)^3
=
4\pi^2 y^{10} \left( y\, y'' - 3 (y')^2\right)^2.
\label{eq:sigma-ode}
\end{equation}
\end{thm}

\begin{proof}
This is a mechanical calculation: substitute
$y=\theta(x)=\frac{\sqrt{2}}{\sqrt{1+x}}\centeredtheta\left(\frac{1+x}{1-x}\right)$ into \eqref{eq:theta-ode} and compute both sides, then simplify algebraically. The details are left to the reader.
\end{proof}

It is useful to rewrite \eqref{eq:sigma-ode} as an equation satisfied by the rescaled version of $\centeredtheta(z)$ defined by
$$ \rescaledcenteredtheta(z) = \frac{1}{\theta(1)} \centeredtheta\left(\frac{z}{\sqrt{\Phi}}\right) = \sum_{n=0} \frac{d(n)}{(2n)!} z^{2n}
\qquad (|z| \le \sqrt{\Phi}),
$$
since $\rescaledcenteredtheta(z)$ is a purely combinatorial generating function whose Taylor expansion contains no transcendental constants. A simple rescaling of \eqref{eq:sigma-ode} yields the following result.

\begin{corr}
The function $\rescaledcenteredtheta(z)$ satisfies the ordinary differential equation
\begin{equation}
\big(
y^2 y''' - 15 y \,
y' \, y'' +
30 (y')^3
\big)^2
+ 32 \left( y\, y'' - 3 (y')^2\right)^3
=
32 y^{10} \left( y\, y'' - 3 (y')^2\right)^2.
\label{eq:sigma-ode-rescaled}
\end{equation}
\end{corr}

\section{Conjectural congruence relations for the $d(n)$'s}

\label{sec:congruences}

Even a casual observation of the sequence $(d(n))_{n=0}^\infty$ reveals that it has interesting arithmetic properties. For example, the congruence relation $d(n)\equiv (-1)^{n-1}\ (\textrm{mod }5)$ is immediately apparent (as an empirical observation). Upon a bit of further inspection, we were led to the following conjectures.

\begin{conj}
\label{conj:cong}
For a finite sequence $(c_1,\ldots,c_k)$, denote by $\overline{(c_1,\ldots,c_k)}$ its periodic extension, and denote $(c_1,\ldots,c_k)^*$ the sequence $(c_1,\ldots,c_k)$ appended by a infinite sequence of zeros.

\medskip
\noindent
(a) The sequence $(d(n))_{n=1}^\infty$ satisfies the following congruences:
\begin{align*}
(d(n) \textrm{ mod }\phantom{0}5)_{n=1}^\infty &= \overline{(1,4)}, \\
(d(n) \textrm{ mod }13)_{n=1}^\infty &= \overline{(
1, 12, 12, 4, 9, 9, 3, 10, 10, 12, 1, 1, 9, 4, 4, 10, 3, 3
)}, \\
(d(n) \textrm{ mod }17)_{n=1}^\infty &= \overline{(
1, 16, 0, 16, 15, 2, 0, 2, 4, 13, 0, 13,}\\& \qquad \ \ \overline{ 9, 8, 0, 8, 16, 1, 0, 1, 2, 15, 0, 15, 13, 4, 0, 4, 8, 9, 0, 9
)},
\\
(d(n) \textrm{ mod }\phantom{0}3)_{n=0}^\infty &= 
(1, 1, 2)^*,
\\
(d(n) \textrm{ mod }\phantom{0}7)_{n=0}^\infty &= 
(1, 1, 6, 2, 2, 2, 1, 0, 3, 0, 6, 0, 6)^*,
\\
(d(n) \textrm{ mod }11)_{n=0}^\infty &= 
(1, 1, 10, 7, 2, 3, 10, 7, 1, 1, 2, 0, 6, 2, 0, 1, 5, 0, \\ & \qquad \ \ 9, 9, 0, 1,
0, 0, 1, 0, 0, 8, 0, 0, 10)^*.
\end{align*}

\noindent
(b) More generally,\,\, for any prime\, $p$ of the form $4k+1$, the sequence $(d(n)\textrm{ mod }p)_{n=1}^\infty$ is periodic, and for any prime $p$ of the form $4k+3$, the sequence has only finitely many nonzero terms.

\end{conj}

Note that in the congruences modulo primes $p=4k+1$ described above, the term $d(0)=1$ is excluded as it does not follow the periodical pattern of the congruence.

\section{Open problems}

\label{sec:openproblems}

We conclude with a few open problems.

\begin{enumerate}

\item Prove Conjecture~\ref{conj:cong}.

\item Extend Conjecture~\ref{conj:cong} further, for example by studying the period of the sequence $(d(n) \textrm{ mod }p)_{n=1}^\infty$ of residues for primes $p=4k+1$, and congruences modulo higher powers of primes.

\item Study integer sequences arising from the Taylor coefficient sequences of other modular forms. Develop a general theory of when such sequences arise and find connections between them and other problems in number theory.

\item Find a combinatorial interpretation for the sequence $(d(n))_{n=0}^\infty$ (the possible existence of such an interpretation is suggested by the use of the combinatorial formula in the proof of Theorem~\ref{thm:recurrence-dofn}).

\item What can be said about the sequence of signs of $d(n)$?

\item The function $\theta_3(x)$ is intimately connected to the theory of the Riemann zeta function via the classical relation
$$
\pi^{-s/2} \Gamma(s/2)\zeta(s) = \frac12 \int_0^\infty (\theta_3(x)-1) x^{s/2-1}\,dx.
$$
By a change of variables $t=(1-x)/(1+x)$ this can be rewritten as
$$
\pi^{-s/2} \Gamma(s/2)\zeta(s) = \int_{-1}^1 \left(\centeredtheta(t)-\frac{1}{\sqrt{1+t}}\right) (1-t)^{s/2-1} (1+t)^{(1-s)/2-1} \, dt.
$$
Can this identity, combined with the Taylor expansion \eqref{eq:sigma-taylor} and additional observations about the coefficient sequence $(d(n))_{n=0}^\infty$, be used to deduce new facts about the Riemann zeta function?

\end{enumerate}

\section*{Appendix: a table of the values $(d(n))_{n=0}^{20}$}

\begin{table}[h!]
$$
\begin{array}{|l|r|}
\hline
n & d(n) \\
\hline
0 & 1
\\
1 & 1
\\
2 & -1
\\
3 & 51
\\
4 & 849
\\
5 & -26199
\\
6 & 1341999
\\
7 & 82018251
\\
8 & 18703396449
\\
9 & -993278479599
\\
10 & -78795859032801
\\
11 & 38711746282537251
\\ 
12 & -923351332174412751
\\
13 & 4688204953344642495801
\\
14 & 501271295036889289819599
\\
15 & -89944302490128540556106949
\\
16 & -104694993963067299023875442751
\\
17 & 63396004159664562363095882996001
\\
18 & -10788308985765935467659682700676801
\\
19 & 8534133600987639916144760846045541651
\\
20 & 16747176493521483129100021404620455570449
\\
%21 & 4491916228699331920281057792897391899437001
%\\
%22 & 4008371490685060692354730557721801155762981999
%\\
\hline
\end{array}
$$
\caption{The initial values of the sequence $(d(n))_{n=0}^\infty$}
\label{table:sequence}
\end{table}

\end{document}